\documentclass[11pt, a4paper]{amsart}

\usepackage{amsfonts}
\usepackage{comment}
\usepackage{amssymb}
\usepackage{mathtools}
\usepackage{amsmath, amsthm,color}
\usepackage{hyperref}
\usepackage{pdfsync}
\usepackage{bbm, dsfont}
\usepackage[margin=0.96in]{geometry}

\newtheorem{theorem}{Theorem}[section]
\newtheorem{corollary}[theorem]{Corollary}
\newtheorem{lemma}[theorem]{Lemma}

\newtheorem*{conjecture}{Conjecture}

\numberwithin{equation}{section}

\RequirePackage[normalem]{ulem} 
\RequirePackage{color}\definecolor{RED}{rgb}{1,0,0}\definecolor{BLUE}{rgb}{0,0,1} 

\begin{document}
\title{On Weissler's conjecture on the Hamming cube I}
\author[P.~Ivanisvili and  F.~Nazarov]{P.~Ivanisvili and   F.~Nazarov}
\address{Department of Mathematics, North Carolina State University, Raleigh, NC, USA}
\email{pivanis@ncsu.edu \textrm{(P.\ Ivanisvili)}}

\address{Department of Mathematical Sciences, Kent State University, Kent, OH, USA}
\email{nazarov@math.kent.edu \textrm{(F.\ Nazarov)}}
\makeatletter
\@namedef{subjclassname@2010}{
  \textup{2010} Mathematics Subject Classification}
\makeatother
\subjclass[2010]{39B62, 42B35, 47A30}
\keywords{}
\begin{abstract} 
Let $1\leq p \leq q <\infty$, and let $w \in \mathbb{C}$.  Weissler conjectured 
that the Hermite operator $e^{w\Delta}$ is bounded as an operator from $L^{p}$ to $L^{q}$ on the Hamming cube $\{-1,1\}^{n}$ with the norm bound independent of $n$ if and only if 
\begin{align*}
|p-2-e^{2w}(q-2)|\leq p-|e^{2w}|q.
\end{align*}
It was proved by Bonami (1970), Beckner (1975),  and Weissler (1979) in all cases except $2<p\leq q <3$ and $3/2<p\leq q <2$, which stood open until now. The goal of this paper is to  give a full proof of Weissler's conjecture in the case $p=q$.  Several applications will be presented.
\end{abstract}
\maketitle

\section{Introduction}
\subsection{Complex hypercontractivity}
Given $n\geq 1$, let $\{-1,1\}^{n}$ be the Hamming cube of dimension $n$, i.e., the set of vectors $x=(x_{1}, \ldots, x_{n})$ such that $x_{j}=1$ or $-1$ for all $j=1, \ldots, n$. For any  $f:\{-1,1\}^{n} \to \mathbb{C}$, define its average value $\mathbb{E}f$ and its $L_{p}$ norm  $\|f\|_{p}$,  $p\geq 1$, to be 
\begin{align*}
\mathbb{E} f  = \frac{1}{2^{n}} \sum_{x \in \{-1,1\}^{n}} f(x) \quad \text{and} \quad \|f\|_{p} = \left(\mathbb{E} |f|^{p}\right)^{1/p}.
\end{align*}
Functions on the Hamming cube can be represented via Fourier--Walsh series. Namely, for any $f:\{-1,1\}^{n} \to \mathbb{C}$, we have 
\begin{align}\label{fur}
f(x) = \sum_{S \subset \{1,\ldots,n\}} a_{S} w_{S}(x), \quad \text{where} \quad w_{S}(x) = \prod_{j \in S}x_{j}
\end{align}
 and $a_{S}$ are the Fourier coefficients of $f$. It follows from (\ref{fur}) that  $a_{S} = \mathbb{E}fw_{S}$ and $a_{\varnothing} = \mathbb{E}f$. For any $z \in \mathbb{C},$ the Hermite\footnote{When $z=e^{-t}$, $t \geq 0$, the traditional notation for the Hermite operator is $e^{-t\Delta}$ instead of $T_{e^{-t}}$. In the quantum field literature $T_{z}$ is called the second quantization operator of $z$. In computer science $T_{z}$ is referred to as the noise operator.} operator $T_{z}$ is defined as
 \begin{align*}
  T_{z}f(x) = \sum_{S \subset \{1,\ldots,n\}} z^{|S|}a_{S} w_{S}(x),
 \end{align*}
 where $|S|$ denotes the cardinality of the set $S \subset \{1,\ldots, n\}$. Weissler~\cite{W1979} made the following
 
 \begin{conjecture}
 Let $1\leq p \leq q <\infty$, and let $z \in \mathbb{C}$. We have 
 \begin{align*}
 \sup_{\|f\|_{p}=1, \, n\geq 1} \|T_{z} f\|_{q}  =  C(p,q,z) <\infty
 \end{align*}
if and only if 
 \begin{align}\label{inf13}
 |p-2-z^{2}(q-2)|\leq p-|z|^{2}q.
 \end{align}
 Moreover, $C(p,q,z)<\infty$ implies $C(p,q,z) = 1$, i.e., that $T_{z}$ is contractive.
 \end{conjecture}

In this paper we prove Weissler's conjecture for $p=q$. We intend to settle the general case in an upcoming manuscript. Our argument for the case $p<q$ requires checking the positivity of two polynomials with large integer coefficients, which is currently hard to present in a human verifiable way.

 
 \newpage

 \subsection{Development of hypercontractivity in mid 70's: Boolean and Gaussian}

 In 1970, Bonami~\cite{Bon1} considered the case $z = r \in \mathbb{R}$. She showed that it is enough to verify  that $T_{r}$ is  contractive  on the Hamming cube of dimension $n=1$.  A little bit earlier, in 1966, Nelson~\cite{Nel1} (independently of Bonami) showed that the Gaussian analog $T_{r}^{G}$ of $T_{r}$  maps boundedly  $L ^{2}(d\gamma)$  to  $L^{q}(d\gamma)$, ($q> 2$, $r \in \mathbb{R}$, $d\gamma$ is the standard Gaussian measure on $\mathbb{R}^{k}$) with norm independent of $k$ provided that  $r$ is sufficiently close to zero. Later, Glimm~\cite{Glim} proved that for  $r$ sufficiently close to $0$, $T^{G}_{r}$ is in fact contractive as an operator  from   $L^{2}(d\gamma)$ to   $L^{q}(d\gamma)$. In 1970, Segal~\cite{Seg} obtained a result which implies that it suffices to check the boundedness of the operator $T^{G}_{r}$  in dimension $k=1$. Finally, in 1973, Nelson~\cite{Nel2}  showed that if $1\leq p \leq q <\infty$, then $T^{G}_{r} : L^{p}(d\gamma) \to L^{q}(d\gamma)$ is contractive if and only if $|r| \leq \sqrt{\frac{p-1}{q-1}}$; and if $|r|>\sqrt{\frac{p-1}{q-1}}$, then $T^{G}_{r}$ is not even bounded. 
 
We should mention that Nelson's result easily follows via the central limit theorem from Bonami's {\em real} hypercontractivity on the Hamming cube.

In 1975, Gross in his celebrated paper \cite{Gross} gave a simple proof of the real hypercontractivity on the Hamming cube by showing its equivalence to log-Sobolev inequalities. Inspired by works of Nelson and Gross\footnote{Apparently not knowing about Bonami.}, Beckner in~\cite{Bec75} obtained the Hausdorff--Young inequality with sharp constants by showing that it follows from the contractivity of $T_{i\sqrt{p-1}}$ from $L^{p}$ to $L^{p'}$, $p'=\frac{p}{p-1}$, on the Hamming cube when $p \in (1,2]$.  At that time, the proofs of the real hypercontractivity by Nelson,  and later by Gross, were real valued, and they could not be extended directly to complex $z$. The main technical part of Beckner's paper \cite{Bec75} is the proof of Bonami's two--point inequality when $z= i\sqrt{p-1}\in \mathbb{C}$, $q=p'$, and $p \in (1,2]$.  

It became an open problem under what conditions on the triples $(p,q,z)$ with $1\leq p\leq q <\infty$ and $z \in \mathbb{C}$ the operator $T^{G}_{z}$ is bounded from  $L^{p}(d\gamma)$ to  $L^{q}(d\gamma)$ with norm independent of  the dimension of the Euclidean space.  In 1979, Coifman, Cwikel, Rochberg, Sagher, and Weiss~\cite{Coif}  proved that $T^{G}_{z}$ is a contraction from $L^{p}(d\gamma)$ to $L^{p'}(d\gamma)$ if $z$ satisfies (\ref{inf13}). The same year,  Weissler~\cite{W1979} proved the full version of the conjecture except when $2<p \leq q<3$, and $3/2<p\leq q <2$.
 Weissler writes in his paper that he believes the theorem should be true for all $1\leq p \leq q<\infty,\,  z \in \mathbb{C}$ that satisfy (\ref{inf13}). The main open problem was to prove that the two-point inequality of Bonami--Beckner 
 \begin{align}\label{tpb}
 \left(\frac{|a+zb|^{q}+|a-zb|^{q}}{2}\right)^{1/q}\leq  \left(\frac{|a+b|^{p}+|a-b|^{p}}{2}\right)^{1/p}
 \end{align}
holds for all  $a,b \in \mathbb{C}$ if the triple $(p,q,z)$ satisfies condition (\ref{inf13}). In 1989, Epperson~\cite{Epperson} proved the Gaussian counterpart of the conjecture by showing that  condition (\ref{inf13}) implies $\|T^{G}_{z}\|_{L^{p}(d\gamma) \to L^{q}(d\gamma)}\leq 1$. His proof avoided the verification of difficult two-point inequalities (\ref{tpb}) and, thereby, did not imply the corresponding result on the Hamming cube.  
 
 In 1990, Lieb obtained a very general theorem~\cite{Lieb1}, which, in particular,  implied the result of Epperson.  After the work of Lieb, in 1997, Janson~\cite{Jan2} gave one more proof of the Gaussian complex hypercontractivity  via Ito calculus.  Janson's argument was later rewritten in terms of heat flows by Hu~\cite{Hu1}. 
Since 1979 no progress has been made on Weissler's conjecture.

\subsection{Applications}

In the case  $p=q>1,$ condition (\ref{inf13}) can be simplified to 
\begin{align*}
\left| z \pm i \frac{|p-2|}{2\sqrt{p-1}}\right| \leq \frac{p}{2\sqrt{p-1}},
\end{align*}
see~(\ref{lens}).
In other words, the admissible domain for $z$ is a lens domain,  i.e., an intersection of two disks (see~Fig.~\ref{fig:dom}). 

\begin{figure}[ht]
\centering
\includegraphics[scale=0.8]{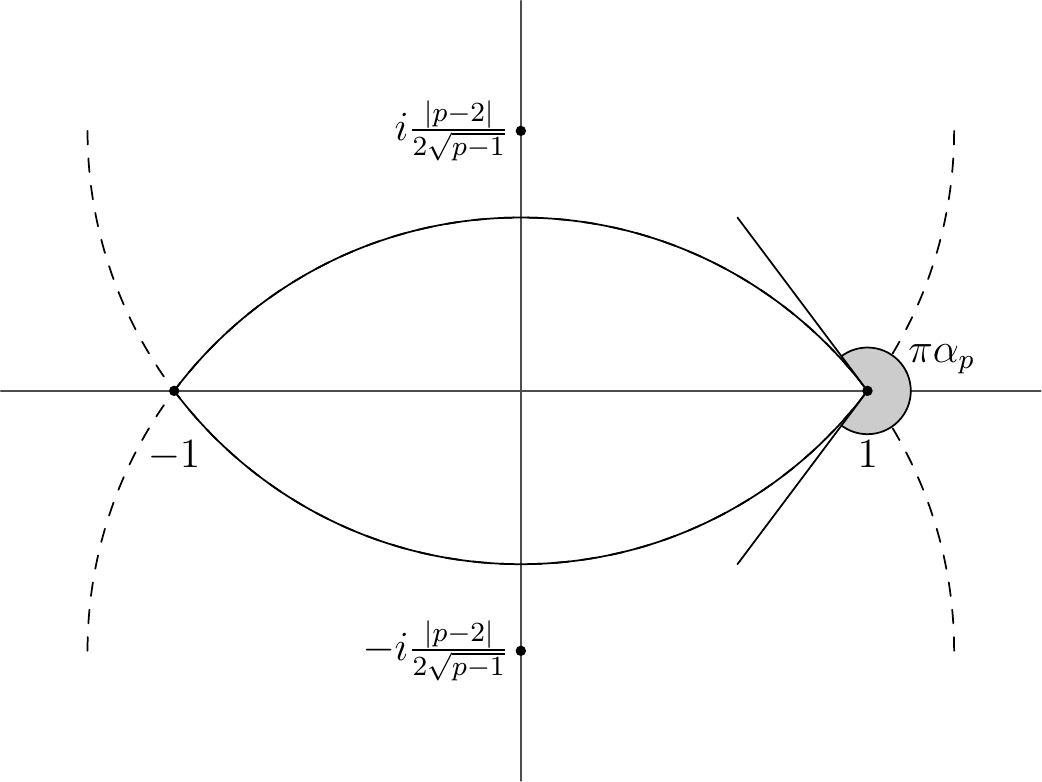}
\caption{Lens domain.}
\label{fig:dom}
\end{figure}

The bounding circles  of these disks pass through the points $z=\pm 1$, which belong to the boundary of the lens domain.  Let $\pi \alpha_{p}$ be the exterior angle  between  the two circles at the point $z=1$. We have 
\begin{align*}
\alpha_{p} = 1+\frac{2}{\pi}\arctan\left(\frac{|p-2|}{2\sqrt{p-1}}\right).
\end{align*}
Recently it was shown in~\cite{IE2} that the angle $\alpha_{p}$ plays an important role in Markov--Bernstein type estimates on the Hamming cube. We will list a series of corollaries that automatically follow from ~\cite{IE2, MN14} and the explicit knowledge of $\alpha_{p}$.  We remind that given 
$f : \{-1,1\}^{n} \to \mathbb{C}$, its Laplacian  $\Delta f $ is defined by $\Delta f= \sum_{j=1}^{n} D_{j} f$ where 
\begin{align*}
D_{j}f(x) = \frac{f(x_{1}, \ldots, x_{j}, \ldots, x_{n}) - f(x_{1}, \ldots, -x_{j}, \ldots, x_{n})}{2}, \quad x = (x_{1}, \ldots, x_{n})\in \{-1,1\}^{n}.
\end{align*}
The operator $\Delta$ is linear, and $\Delta w_{S}(x)=|S|w_{S}(x)$. Also define the discrete gradient 
\begin{align*}
|\nabla f|^{2} = \sum_{j=1}^{n} (D_{j}f)^{2}.
\end{align*}
\begin{corollary}\label{cari}
For each $p>1$, there exist finite $c_{1},c_{2}, c_{3}>0$ depending on $p$  such that for all  $f  = \sum_{S\subset \{1,\ldots, n\}, \; |S|\geq  d} a_{S}w_{S}$ and all $n\geq d$, we have 
\begin{align}\label{manor1}
\|e^{-t\Delta}f\|_{p} \leq c_{1} e^{-c_{2} d \min\{t,t^{\frac{1}{2-\alpha_{p}}}\}}\|f\|_{p} \quad \text{for all} \quad  t\geq 0,
\end{align}
and hence
\begin{align}\label{manor2}
\|\Delta  f\|_{p} \geq  c_{3} d^{2-\alpha_{p}} \|f\|_{p}.
\end{align}
\end{corollary}
In \cite{MN14} Mendel and Naor showed that if  $f$ takes values in a K-convex Banach space $X$, i.e.,  $a_{S} \in X$ for all $S \in \{1, \ldots, n\}$,  then inequality (\ref{manor1}) holds with some $\alpha_{p}(X) \in (0,1)$. They asked  if  $\alpha_{p}(X)$ in the inequalities  (\ref{manor1}) and (\ref{manor2}) can be replaced by $1$. The question is open even for $X=\mathbb{R}$  and is known as the ``heat smoothing conjecture'' (see \cite{HMO1}, where the case $d=1$ has been resolved, and \cite{IE2}, where the conjecture has been resolved for functions  with``narrow'' spectrum). In \cite{IE2} the first author with A.~Eskenazis showed that the conclusion of Corollary~\ref{cari} holds for $p\in (1,3/2)\cup (3, \infty)$ thanks to the theorem of Weissler~\cite{W1979}. Repeating the arguments in \cite{IE2} verbatim and using the main result of this paper for $p=q \in (3/2,3)$ we obtain Corollary~\ref{cari} for all $p \in (1, \infty)$. 

 We know that the bounds (\ref{manor1}) and (\ref{manor2}) are not sharp in general; for example when $p>2$ is such that $\alpha_{p}>3/2$, i.e., $p>4+2\sqrt{2}$, then better bounds are available due to Meyer.
\begin{theorem}[Meyer~\cite{meyer11}]
For each $p\geq 2$, there exist $C_{p}, c_{p}>0$  such that for any  $f  = \sum_{S\subset \{1,\ldots, n\}, \; |S|\geq  d} a_{S}w_{S}$ and all $n\geq d$, we have 
\begin{align*}
\|e^{-t\Delta}f\|_{p} \leq e^{-c_{p} d \min\{t,t^{2}\}}\|f\|_{p} \quad \text{for all} \quad  t\geq 0, 
\end{align*}
and hence 
\begin{align*}
\|\Delta  f\|_{p} \geq  C_{p} d^{1/2} \|f\|_{p}.
\end{align*}
\end{theorem}

The next corollary is due to Eskenazis--Ivanisvili \cite{IE2}, proved for $p\in (1,3/2)\cup (3, \infty)$ and now valid (by the main result of this paper) for all $p \in  (1, \infty)$. It naturally extends Freud's inequalities on the Hamming cube

\begin{corollary}\label{FFR}
For all $p>1,$ there exists finite $C_{p}>0$ such that 
\begin{align*}
&\|\Delta  f\|_{p} \leq 10 d^{\alpha_{p}} \|f\|_{p}, \\
&\|\nabla f\|_{p}\leq
\begin{cases}
C_{p} d^{\alpha_{p}/p}\ln(d+1) \|f\|_{p},  \quad p \in (1,2),\\
C_{p} d^{\alpha_{p}/2}\|f\|_{p},  \quad p \in [2, \infty],
\end{cases} 
\end{align*}
for all $f  = \sum_{S\subset \{1,\ldots, n\}, \; |S|\leq d} a_{S}w_{S}$ and all $n\geq d$. 
\end{corollary}

\subsubsection{Complex and real hypercontractivity} To illustrate the advantage of complex hypercontractivity over the real one  let us outline a method described in \cite{IE2} for obtaining bounds on the norms of Fourier multipliers on spaces of functions on the Hamming cube with restricted spectrum  in a systematic way.

Let $1<p\leq q$, and suppose we are interested in obtaining a bound of the type 
\begin{align}\label{furfur}
\Biggl\| \sum_{S \subset \{0, \ldots, n\}, \,  |S|\leq d} \varphi(|S|) a_{S} w_{S} \Biggr\|_{q} \leq C_{p,q, d,  \varphi} \Biggl\| \sum_{S \subset \{0, \ldots, n\}, \,  |S|\leq d} a_{S} w_{S} \Biggr\|_{p}
\end{align}
where  $\varphi$ is some fixed function, which is usually called a Fourier multiplier.  The complex hypercontractivity yields the bound 
\begin{align}\label{tochki}
\Biggl\| \sum_{S \subset \{1, \ldots, n\}, \,  |S|\leq d} z^{|S|} a_{S} w_{S} \Biggr\|_{q} \leq \Biggl\| \sum_{S \subset \{1, \ldots, n\}, \,  |S|\leq d} a_{S} w_{S} \Biggr\|_{p},
\end{align}
which holds true for all $z \in \Omega_{p,q}$, where $\Omega_{p,q}$ is the domain described in (\ref{inf13}). To obtain the bound of the type (\ref{furfur}) it suffices to find a complex valued measure $\mu$ on $\Omega_{p,q}$ such that 
\begin{align*}
\int_{\Omega_{p,q}} z^{j} d\mu(z) = \varphi(j) \quad \text{for all} \quad j=0,\ldots, d. 
\end{align*} 
Then the triangle inequality together with (\ref{tochki}) gives 

\begin{align*}
&\Biggl\| \sum_{S \subset \{1, \ldots, n\}, \,  |S|\leq d} \varphi(|S|) a_{S} w_{S} \Biggr\|_{q} \leq   \int_{\Omega_{p,q}} \Biggl\| \sum_{S \subset \{1, \ldots, n\}, \,  |S|\leq d} z^{|S|} a_{S} w_{S} \Biggr\|_{q} d|\mu|(z) \stackrel{(\ref{tochki})}{\leq}\\
& \| \mu\|  \Biggl \| \sum_{S \subset \{1, \ldots, n\}, \,  |S|\leq d} a_{S} w_{S} \Biggr\|_{p},
\end{align*}
where $\| \mu\|$ stands for the total variation norm of $\mu$.

To minimize the norm of $\mu$, we may invoke the Hahn--Banach theorem together with the Riesz representation theorem. Let  $\mathcal{P}^{d} \subset C(\Omega_{p,q})$ be a subspace  consisting of all analytic polynomials of degree $d$. Let $L$ be a linear functional on 
$\mathcal{P}^{d}$ such that  $L(z^{j})=\varphi(j)$ for all $0\leq j \leq d$.  Clearly its norm $\| L\|$  is the smallest positive constant $C_{p,q,d,\varphi}$ for which  
\begin{align}\label{expr1}
\Biggl| \sum_{j=0}^{d} \varphi(j)a_{j} \Biggr| \leq C_{p,q,d,\varphi}\,    \Biggl\| \sum_{j=0}^{d}  a_{j}z^{j}\Biggr\|_{C(\Omega_{p,q})} \quad \text{for all} \quad a_{0}, ..., a_{d} \in \mathbb{C}. 
\end{align}
By the Hahn--Banach theorem, there exists $\widetilde{L} \in C^{*}(\Omega_{p,q})$ such that $\widetilde{L}|_{\mathcal{P}^{d}} = L$, and $\| \widetilde{L}\|_{C^{*}(\Omega_{p,q})} = C_{p,q,d,\varphi}$. By the Riesz representation theorem $\widetilde{L}(h)  = \int_{\Omega_{p,q}} h(z) d\mu(z)$ for some complex valued Radon measure $d\mu$ on $\Omega_{p,q}$ such that $\int_{\Omega_{p,q}}d |\mu(z)| = \| \widetilde{L}\|_{C^{*}(\Omega_{p,q})} = C_{p,q,d,\varphi}$. Thus we obtain (\ref{furfur})  with the constant $C_{p,q,d,\varphi}$ that solves the extremal problem (\ref{expr1}). 

If we use only the real hypercontractivity in this argument, then the space $C(\Omega_{p,q})$  on the right hand side of (\ref{expr1}) will be replaced by $C(\Omega_{p,q}^{\mathbb{R}})$ where 
$$
\Omega_{p,q}^{\mathbb{R}} = \left[-\sqrt{\frac{p-1}{q-1}}, \sqrt{\frac{p-1}{q-1}}\right] \subset \mathbb{R},
$$ 
in which case one gets (\ref{expr1}) with a constant $C_{p,q,d,\varphi}^{\mathbb{R}} \geq C_{p,q,d,\varphi}$ due to the fact that $\Omega_{p,q}^{\mathbb{R}} \subset \Omega_{p,q}$. 

Thus we see that ``the amount of improvement'' the complex hypercontractivity gives over the real one is determined by how small the norm  of the functional $L$ on $\mathcal{P}^{d}$ equipped with the $C(\Omega_{p,q})$ norm can be compared to the case when the norm on  $\mathcal{P}^{d}$  is replaced by the $C(\Omega^{\mathbb{R}}_{p,q})$ one.

Some more applications of complex hypercontractivity are given in \cite{IE2}.

\section{The proof}

We start with several observations (Sections~\ref{oerti} -- \ref{oxuti}) that are well known to experts.   We decided to include them here for the reader's convenience.  

\subsection{Tensor power trick  and induction on dimension}\label{oerti}
In this section we show the equivalence of  several inequalities. 
\begin{lemma}\label{equi1}
Let $1\leq p \leq q<\infty$ and $z \in \mathbb{C}$ be fixed. The following  are equivalent:
\begin{itemize}
\item[(i)] $\|T_{z} f\|_{q} \leq C(p,q,z)\|f\|_{p}$ for some $C(p,q,z)<\infty$,  all $f:\{-1,1\}^{n} \to \mathbb{C}$, and all $n\geq 1$. 
\item[(ii)] $\|T_{z} f\|_{q} \leq \|f\|_{p}$ for all $f:\{-1,1\}^{n} \to \mathbb{C}$ and all $n\geq 1$.
\item[(iii)] $\|T_{z} f\|_{q} \leq \|f\|_{p}$ for all $f:\{-1,1\} \to \mathbb{C}$.
\end{itemize} 
\end{lemma} 
Clearly it follows from the lemma  that the conjecture will be  proved once the equivalence of the two-point inequality (iii) and the condition (\ref{inf13}) is verified.  
\begin{proof}
Obviously (ii) implies (i). To show that (i) implies (ii), consider $F(x) = f(x^{1})\cdots f(x^{k})$ where $x=(x^{1}, \ldots, x^{k})\in \{-1,1\}^{nk}$. Then $\|F\|_{p} = \|f\|_{p}^{k}$, $\|T_{z}F\|_{q} = \|T_{z} f\|_{q}^{k}$. Therefore, (i) implies $\|T_{z}F\|_{q} \leq C(p,q,z)\|F\|_{p}$, which in turn implies $\|T_{z}f\|_{q}\leq (C(p,q,z))^{1/k}\|f\|_{p}$. Letting $k \to \infty$ we obtain (ii). 

Obviously (ii) implies (iii). Next, we  show that (iii) implies (ii).  Let 
\begin{align}\label{multiv}
f(x_{1}, \ldots, x_{n}) = \sum_{S \subset \{1,\ldots, n\}}a_{S} \prod_{j \in S} x_{j}.
\end{align}
 Let us extend the domain of definition of $f$ to be all $\mathbb{R}^{n}$ by considering the right hand side of (\ref{multiv})  as a multivariate polynomial of variables $x_{1}, \ldots, x_{n}$.  Then we can write
\begin{align*}
T_{z}f(x_{1}, \ldots, x_{n}) = f(zx_{1}, \ldots, zx_{n}) \quad \text{for all} \quad (x_{1}, \ldots, x_{n}) \in \{-1,1\}^{n}.
\end{align*}
By (iii), for any complex numbers $A,B \in \mathbb{C}$, we have 
\begin{align}\label{twov1}
\mathbb{E}_{x_{1}} |A+zx_{1} B|^{q} \leq \left(\mathbb{E}_{x_{1}} |A+x_{1} B|^{p}\right)^{q/p}
\end{align}
where $\mathbb{E}_{x_{1}}$ means that we take the expectation with respect to the symmetric $\pm1$ Bernoulli random variable $x_{1}$. Since 
\begin{align*}
f(zx_{1}, zx_{2}, \ldots, zx_{n}) = A(zx_{2}, \ldots, zx_{n}) + zx_{1} B(zx_{2}, \ldots, zx_{n}),
\end{align*} 
we can write 
\begin{align*}
&\|T_{z}f\|^{p}_{q} =\left(\mathbb{E}_{x_{n}}\ldots \mathbb{E}_{x_{1}} |A(zx_{2}, \ldots, zx_{n}) + zx_{1} B(zx_{2}, \ldots, zx_{n})|^{q}  \right)^{p/q}\stackrel{(\ref{twov1})}{\leq} \\
&\left(\mathbb{E}_{x_{n}}\ldots \mathbb{E}_{x_{2}} \left( \mathbb{E}_{x_{1}} |A(zx_{2}, \ldots, zx_{n}) + x_{1} B(zx_{2}, \ldots, zx_{n})|^{p}\right)^{q/p}  \right)^{p/q}\stackrel{\mathrm{Minkowski}}{\leq} \\
&\mathbb{E}_{x_{1}} \left(\mathbb{E}_{x_{n}}\ldots \mathbb{E}_{x_{2}} |A(zx_{2}, \ldots, zx_{n}) + x_{1} B(zx_{2}, \ldots, zx_{n})|^{q}  \right)^{p/q} \stackrel{\mathrm{induction}}{\leq}  \mathbb{E} |f|^{p}.
\end{align*}
Notice that in the second inequality we used the condition $1\leq q/p$. 
\end{proof}

In the next section we explain where the condition (\ref{inf13}) comes from.

\subsection{From global to local: the necessity part}\label{oori}

The condition (iii) of Lemma~\ref{equi1} is equivalent to the following two-point inequality 
\begin{align}\label{two-pointa}
\left( \frac{|1+wz|^{q} + |1-wz|^{q}}{2}\right)^{1/q} \leq\left( \frac{|1+w|^{p}+|1-w|^{p}}{2}\right)^{1/p},
\end{align}
which should hold true for all $w \in \mathbb{C}$. Let us explain that condition (\ref{inf13}) is in fact an ``infinitesimal'' form of the inequality (\ref{two-pointa}) when $w \to 0$. Indeed, let $w = \varepsilon v$ where $v \in \mathbb{C}$, $|v|=1$, is fixed and $\varepsilon >0$. As $\varepsilon \to 0$, we have 
\begin{align*}
|1+w|^{p} &= (1+2\varepsilon  \Re v + \varepsilon^{2} |v|^{2})^{p/2}  \\
&=1 + \frac{p}{2}\left( 2\varepsilon \Re v + \varepsilon^{2}|v|^{2}\right) +\frac{p}{4}\left(\frac{p}{2}-1\right)4 \varepsilon^{2} (\Re v)^{2} + O(\varepsilon^{3})\\
&=1+\varepsilon p \Re v + \varepsilon^{2} \frac{p}{2}\left(|v|^{2} + (p-2)(\Re v)^{2} \right) + O(\varepsilon^{3}).
\end{align*}

Therefore, comparing the second order terms, we see that the two-point inequality (\ref{two-pointa}), in particular, implies that 
\begin{align}\label{pequ}
|vz|^{2} + (q-2) (\Re vz)^{2} \leq |v|^{2} + (p-2)(\Re v)^{2} \quad \text{for all unit vectors} \quad v \in \mathbb{C}.
\end{align}
The last inequality can be rewritten as 
\begin{align*}
(p-2)\left(\frac{v+\bar{v}}{2}\right)^{2} - (q-2)\left(\frac{vz+\overline{vz}}{2}\right)^{2}\geq |z|^{2} -1.
\end{align*}
Multiplying by $2$ and opening the parentheses, we obtain 
\begin{align*}
(p-2)-(q-2)|z|^{2} + \Re [((p-2)-(q-2)z^{2})v^{2}] \geq 2|z|^{2}-2,
\end{align*}
i.e., 
\begin{align*}
-\Re [((p-2)-(q-2)z^{2})v^{2}]  \leq p-q|z|^{2},
\end{align*}
which, since $v$ is an arbitrary unit vector, is equivalent to (\ref{inf13}).

\subsection{The $\mathrm{inf}$ representation}\label{osami}
It will be helpful to describe the domain $\Omega_{p,q}$ of all $z$'s satisfying (\ref{inf13}) in polar coordinates. Notice that if $q=1$, then $p=1$ and, therefore, (\ref{inf13}) implies that $z \in [-1,1]$. In this case (\ref{tpb}) trivially holds by convexity. In what follows we assume that $q>1$.

Let $z=re^{it} \in \Omega_{p,q}$ and $v=e^{i \beta}$. Then (\ref{pequ}) takes the form 
\begin{align*}
r^{2} \left( 1 + (q-2)\cos^{2}(t+\beta)\right) \leq 1+(p-2)\cos^{2}(\beta).
\end{align*}
  Dividing both sides of  this inequality  by $1 + (q-2)\cos^{2}(t+\beta)$ and taking the infimum over all $\beta \in \mathbb{R}$, we obtain 
\begin{align*}
r \leq \inf_{\beta\in \mathbb{R}} \sqrt{\frac{1+(p-2)\cos^{2}(\beta)}{1+(q-2)\cos^{2}(t+\beta)}}.
\end{align*}
Given $t \in \mathbb{R}$, let $z=r(t) e^{it}$, $r(t)\geq 0$ be such that  $z$ lies on the boundary of $\Omega_{p,q}$. Then 
\begin{align}\label{temp08}
r(t) = \inf_{\beta\in \mathbb{R}}  \sqrt{\frac{1+(p-2)\cos^{2}(\beta)}{1+(q-2)\cos^{2}(t+\beta)}}.
\end{align}
It follows from (\ref{temp08}) that $r(t)$ is an even $\pi$-periodic function.


Throughout  the rest of the paper we assume that $p=q>1$.  

\subsection{Lens domain}\label{ootxi}

For $p=q$ the two-point inequality (\ref{two-pointa}) takes the form 
 \begin{align}\label{two-point}
|1+wz|^{p} + |1-wz|^{p} \leq |1+w|^{p}+|1-w|^{p} \quad \text{for all} \quad w \in \mathbb{C},
\end{align}
and (\ref{inf13}) takes the form 
\begin{align}\label{nozhnice1}
1-|z|^{2} \geq \frac{|p-2|}{p}|1-z^{2}|.
\end{align}
 Squaring  and subtracting  $(1-|z|^{2})^{2} \frac{(p-2)^{2}}{p^{2}}$ from both sides of the inequality (\ref{nozhnice1}), we obtain 
\begin{align*}
(1-|z|^{2})^{2} \frac{4(p-1)}{p^{2}} \geq \frac{(p-2)^{2}}{p^{2}} \left(|1-z^{2}|^{2} - (1-|z|^{2})^{2} \right) = \frac{(p-2)^{2}}{p^{2}} \left(\frac{z-\bar{z}}{i} \right)^{2}.
\end{align*}
From (\ref{nozhnice1}) we see that $|z|\leq 1$, so taking the square root in the last inequality we obtain 
\begin{align*}
1-|z|^{2} \geq \frac{|p-2|}{\sqrt{p-1}}|\Im z|.
\end{align*}
The last condition can be rewritten as 
\begin{align}\label{lens}
\left| z \pm i \frac{|p-2|}{2\sqrt{p-1}} \right|\leq \frac{p}{2\sqrt{p-1}}.
\end{align}
The set (\ref{lens}) represents a ``lens'' domain, i.e.,  the intersection of two discs centered at points $\pm \frac{i|p-2|}{2\sqrt{p-1}}$ of radii $\frac{p}{2\sqrt{p-1}}$ whose boundary circles pass through the points $1$ and $-1$.

Recall that the boundary of $\Omega_{p,p}$ is described in polar coordinates by the equation $z = r(t)e^{it}$. It is easy to see from (\ref{lens}) that  $r(t)$ is a decreasing function on $[0, \pi/2]$, 
\begin{align}\label{defr}
r(0)=1,  \quad  r(\pi/2)=\min\left\{\sqrt{p-1}, \frac{1}{\sqrt{p-1}}\right\}.
\end{align}
In the next section we explain that it is enough to prove the two-point inequality (\ref{two-point}) only for $p\geq 2$ and a certain family of points $z$ and $w$. 
\subsection{Duality, symmetry, and convexity}\label{oxuti}
Since $(T_{z})^{*} = T_{\bar{z}}$, and $\| T_{z} \|_{L^{p} \to L^{p}} = \| (T_{z})^{*} \|_{L^{p'} \to L^{p'}}$  (as usual $p'=\frac{p}{p-1}$), we may assume without loss of generality that $p\geq 2$. The reader may also verify that the condition (\ref{lens}) is invariant under the replacement of $p$ by $p'$.

Next, we claim that it suffices to check (\ref{two-point}) for $z \in \partial \Omega_{p,p}$.  Indeed,  every interior point $z \in \Omega_{p,p}$ can be written as a convex combination of two boundary points  $z_{1}, z_{2}$ of $\Omega_{p,p}$. Since the function 
\begin{align*}
z \mapsto \frac{|1+wz|^{p}+|1-wz|^{p}}{2}
\end{align*} 
is convex on $\mathbb{C}$,  its value at $z$ does not exceed the maximum of its values at $z_{1}$ and $z_{2}$.

In what follows, we set $z=r(t)e^{it}$, $t\in \mathbb{R}$. Let us rewrite (\ref{two-point}) as 
\begin{align}\label{two-point1}
\left|1+\frac{w}{z} \right|^{p} +\left|1-\frac{w}{z} \right|^{p}  \geq |1+w|^{p}+|1-w|^{p}.
\end{align}
Let 
\begin{align}\label{maxc}
c(t) = \frac{1}{r(t)} \in \left[1, \sqrt{p-1}\right] \quad (\text{recall that}\; \; p\geq 2)
\end{align}
 and let $w=y e^{ia}$ with $y\geq 0, a\in \mathbb{R}$. Notice that $w/z = c(t) y e^{i(-t+a)} = c(-t) y e^{i(-t+a)}$. Changing the variable $-t$ back to $t$, we can rewrite the  inequality (\ref{two-point1})  as follows
\begin{align}\label{in1}
|1+c(t)ye^{i(t+a)}|^{p} + |1-c(t)ye^{i(t+a)}|^{p}\geq |1+ye^{ia}|^{p}+|1-ye^{ia}|^{p}.
\end{align}
\pagebreak
\begin{lemma}\label{dadis}
It is enough to check  (\ref{in1}) for $0\leq a \leq a+t \leq \frac{\pi}{2}$. 
\end{lemma}
\begin{proof}
Denote for brevity $c=c(t)$, and rewrite (\ref{in1}) as 
\begin{align}\label{in3}
\left(c^{2}y^{2}+1+2cy\cos(a+t)\right)^{p/2}&+\left(c^{2}y^{2}+1 - 2cy \cos(a+t)\right)^{p/2}\geq \\
\left(y^{2}+1+2y\cos(a)\right)^{p/2}&+\left(y^{2}+1-2y\cos(a)\right)^{p/2}. \nonumber
\end{align}
The map $s \mapsto  |1+sye^{i(t+a)}|^{p}$ is convex. Therefore the map 
$$
s \mapsto |1+sye^{i(t+a)}|^{p}+|1-sye^{i(t+a)}|^{p} = \left(s^{2}y^{2}+1+2sy\cos(a+t)\right)^{p/2}+\left(s^{2}y^{2}+1 - 2sy \cos(a+t)\right)^{p/2}
$$
 is increasing for $s\geq 0$. Since $c\geq 1$, we have  
\begin{align*}
\left(c^{2}y^{2}+1+2cy\cos(a+t)\right)^{p/2}&+\left(c^{2}y^{2}+1 - 2cy \cos(a+t)\right)^{p/2}\geq \\
\left(y^{2}+1+2y\cos(a+t)\right)^{p/2}&+\left(y^{2}+1 - 2y \cos(a+t)\right)^{p/2}. 
\end{align*}
Also notice that since $p\geq 2$, the map $s \mapsto (A+Bs)^{p/2}$ is convex and, thereby, the map  $s \mapsto (A+Bs)^{p/2}+(A-Bs)^{p/2}$ is increasing for $s\geq 0$ as long as $A\pm Bs \geq 0$.  Thus,  if $|\cos(a+t)|\geq |\cos(a)|$, then
\begin{align*}
\left(y^{2}+1+2y\cos(a+t)\right)^{p/2}&+\left(y^{2}+1 - 2y \cos(a+t)\right)^{p/2}\geq \\
\left(y^{2}+1+2y\cos(a)\right)^{p/2}&+\left(y^{2}+1 - 2y \cos(a)\right)^{p/2},
\end{align*}
i.e., inequality (\ref{in1}) trivially holds whenever $|\cos(a+t)|\geq |\cos(a)|$.

  By the $2\pi$-periodicity of $\cos(x)$ and $c(t)$ we can assume that $a, t \in [0, 2\pi)$. 
  
$\textup{(i)}$  Suppose  $a \in [0, \pi/2]$.  The assumption $|\cos(a+t)|< |\cos(a)|$ implies that $a+t \in [a, \pi-a] \cup [\pi+a, 2\pi-a]$. Consider first the case when $a+t \in [a, \pi-a]$.   If $a+t \in [\pi/2, \pi-a]$, then  $t^{*} = \pi -t-2a \geq 0$. Clearly $t^{*}$ satisfies $a+t^{*} \leq \pi/2$,   $|\cos(a+t)| = |\cos(a+t^{*})|$.  Since  $t, t+2a \in [0, \pi]$ and $t+a \geq \pi/2$, the inequality $|\pi/2 -t|\leq |\pi/2 -(2a+t)|\leq \pi/2$ holds and, thereby,  $c(t^{*}) = c(t+2a) \leq c(t)$ (we remind that $c(s)=1/r(s)$, $s \mapsto r(s)$ is decreasing on $[0,\pi/2]$ and $r(s)=r(\pi-s)$ for $s \in [0,\pi/2]$, so the closer $s \in [0, \pi]$ is to $\pi/2$, the larger $c(s)$ is). 
  Thus (\ref{in3}) becomes stronger if we replace $t$ by $t^{*}$.

If $a+t\in [a+\pi, 2\pi-a]$, then we consider $t^{*} = t-\pi$ and use the fact that $c(t^{*}) = c(t)$ and $|\cos(a+t)| = |\cos(a+t^{*})|$. On the other hand, $a+t^{*} \in [a, \pi-a]$, which reduces this subcase to the previously considered one. 

$\textup{(ii)}$ Suppose $a \in (\pi/2, \pi)$. Consider $a^{*} = \pi-a \in (0, \pi/2)$ and $t^{*} = 2\pi -t$. 
Then $|\cos(a^{*})|=|\cos(a)|$,  $|\cos(a^{*}+t^{*})| =|\cos(a+t)|$,  $c(t) = c(t^{*})$, and the inequality reduces to case $\textup{(i)}$. 

$\textup{(iii)}$ Suppose $a \in [\pi, 2\pi)$. Then replace $a$ by  $a^{*} = a - \pi$ and  reduce the inequality in question to the previous two cases.  The lemma is proved. 
\end{proof}

\subsection{Proof of (\ref{in3}) when $p\geq 3$ via ``mock log-Sobolev inequality''}\label{weiapp}
Let us give a proof of (\ref{in3}) for $p\geq 3$. This case was proved by Weissler~\cite{W1979}. His argument  is similar to the proof of the equivalence of the log-Sobolev inequality and the real  hypercontractivity. Indeed, let us briefly mention the connection. The real hypercontractivity is equivalent (see \cite{Bon1}) to the following two-point inequality
\begin{align}\label{realh}
\left(\frac{ \left| a+\sqrt{\frac{p-1}{q-1}}\, b\right|^{q} + \left| a - \sqrt{\frac{p-1}{q-1}}\, b\right|^{q}}{2}\right)^{1/q} \leq \left( \frac{\left| a+ b\right|^{p} + \left| a - b\right|^{p}}{2}\right)^{1/p}
\end{align}   
for all $1<p\leq q <\infty$ and all $a,b \in \mathbb{R}$. The factor $\sqrt{\frac{p-1}{q-1}}$ is a ratio of the values of the same function $s \mapsto \sqrt{s-1}$ at $s=p$ and $s=q$.  Therefore   (\ref{realh}) is equivalent to the statement that the mapping 
\begin{align}\label{logs}
p \mapsto \left( \frac{\left| a+ \frac{x}{\sqrt{p-1}}\right|^{p} + \left| a - \frac{x}{\sqrt{p-1}}\right|^{p}}{2}\right)^{1/p}
\end{align}
is decreasing on $(1, \infty)$  for all fixed $a,x \in \mathbb{R}$. Differentiating (\ref{logs}) with respect to $p$, we arrive at what is called  ``the log-Sobolev inequality on the two-point space'' (see ~\cite{Gross}).  

Ideally, we would like to use the same idea when proving (\ref{in3}). The first obstacle is that the  factor $c(t)$ does not have the  desired quotient structure. Nevertheless, using  (\ref{maxc}) and the representation (\ref{temp08}), we can fix this issue estimating $c(t)$ from below as  
\begin{align}\label{infl}
c(t)  = \sup_{\beta \in \mathbb{R}} \frac{\sqrt{1+(p-2)\cos^{2}(t+\beta)}}{\sqrt{1+(p-2)\cos^{2}(\beta)}} \stackrel{\beta = -a-t}{\geq} \frac{\sqrt{1+(p-2)\cos^{2}(a)}}{\sqrt{1+(p-2)\cos^{2}(a+t)}}.
\end{align}
Therefore,  (\ref{in3})  with fixed $p>2$ is implied by (but no longer equivalent to) the statement that the mapping  
\begin{align}\label{map1}
s \mapsto &\left(1+\frac{x^{2}}{1+(p-2)\cos^{2}(s)}  + \frac{2x \cos(s)}{\sqrt{1+(p-2)\cos^{2}(s)}} \right)^{p/2} \\
+ &\left(1+\frac{x^{2}}{1+(p-2)\cos^{2}(s)}  - \frac{2x \cos(s)}{\sqrt{1+(p-2)\cos^{2}(s)}} \right)^{p/2}\nonumber
\end{align}
is increasing on $[0, \pi/2]$ for all $x \geq 0$ (and the fact that for $p>2$ the left hand side of (\ref{in3}) is increasing in $c$). This statement  seems to be a right substitute for the log-Sobolev inequality in the complex contractivity case. The next lemma shows that,  unfortunately, this monotonicity holds only for $p\geq 3$. 
\begin{lemma}[``Mock log-Sobolev inequality'']\label{imlog}
Let $p>2$. The map (\ref{map1}) is increasing on $[0, \pi/2]$ for all fixed $x \in \mathbb{R}$ if and only if  $p \geq 3$. 
\end{lemma}

\begin{proof}
 Denote $b = \frac{1}{1+(p-2)\cos^{2}(s)} \in [\frac{1}{p-1},1]$ and $u=\frac{x}{\sqrt{p-2}} \in \mathbb{R}$. We want to show that the map 
 \begin{align*}
 b \mapsto \left(1+bu^{2}(p-2)+2u \sqrt{1-b}\right)^{p/2}+\left(1+bu^{2}(p-2)-2u \sqrt{1-b}\right)^{p/2}
 \end{align*}
 is increasing on $[\frac{1}{p-1},1]$. Without loss of generality, assume $u \geq 0$. After taking the derivative with respect to $b$, we end up with showing that  
 \begin{align}\label{we2}
\left(\frac{1+bu^{2}(p-2)-2u\sqrt{1-b}}{1+bu^{2}(p-2)+2u\sqrt{1-b}}\right)^{p/2-1} -\frac{1-u(p-2)\sqrt{1-b}}{1+u(p-2)\sqrt{1-b}}\geq 0.
 \end{align}
 If $b=\frac{1}{p-1}$, then  (\ref{we2}) takes the form 
 \begin{align}\label{mmk1}
  \left| \frac{1-u \sqrt{\frac{p-2}{p-1}}}{1+u \sqrt{\frac{p-2}{p-1}}}\right|^{p-2}- \frac{1 - u\sqrt{\frac{p-2}{p-1}} (p-2)}{1+u\sqrt{\frac{p-2}{p-1}} (p-2)} \geq 0.
 \end{align}
 Denote $u \sqrt{\frac{p-2}{p-1}}=k \geq 0$ and $p-2=\alpha\geq 0$. If $\alpha \in (0,1)$ and $k \in (0,1)$, then   
 \begin{align}\label{ax01}
 \left|\frac{1-k}{1+k}\right|^{\alpha} \leq \frac{1-k\alpha}{1+k\alpha}.
 \end{align}
 Indeed, it follows from the following general principle: if $a_{k}\geq 0$ and the function $g(x) = 1-\sum_{k \geq 1} a_{k} x^{k}>0$ on  $(-1,1)$, then 
 \begin{align*}
 \frac{g(x)}{g(-x)} = \frac{1-a_{1}x-\sum_{k \geq 2} a_{k} x^{k}}{1+a_{1}x-\sum_{k \geq 2} a_{k} (-1)^{k}x^{k}} \leq \frac{1-a_{1}x}{1+a_{1}x}, \quad x \in [0,1)
 \end{align*}
 because of the inequality $\sum_{k \geq 2} a_{k} x^{k} \geq \sum_{k \geq 2} a_{k} (-1)^{k}x^{k}$ and the fact that the mapping $s \mapsto  \frac{A+s}{B+s}$ is increasing when $A \leq B$ and both the numerator and the denominator are nonnegative. Note that $g(k)=(1-k)^{\alpha}$ satisfies the assumptions of  this principle when $\alpha \in (0,1)$. Thus for $p\in (2,3)$ we obtain the inequality which is reverse to (\ref{mmk1}).

Let now that $p\geq 3$. In (\ref{we2}) we can assume that $u\geq 0$ is such that $1-u(p-2)\sqrt{1-b}>0$, otherwise there is nothing to prove. We have 
\begin{align*}
\left[ \left(\frac{1-u(p-2)\sqrt{1-b}}{1+u(p-2)\sqrt{1-b}}\right)^{\frac{1}{p-2}}\right]^{2} \stackrel{(\ref{ax01})}{\leq} \left[\frac{1-u\sqrt{1-b}}{1+u\sqrt{1-b}}\right]^{2} &=\frac{1-2u\sqrt{1-b}+u^{2}(1-b)}{1+2u\sqrt{1-b}+u^{2}(1-b)} \\
& \leq \frac{1-2u\sqrt{1-b}+bu^{2}(p-2)}{1+2u\sqrt{1-b}+bu^{2}(p-2)},
\end{align*}
where in the last inequality  we used the above observation about the monotonicity of the mapping $s \mapsto  \frac{A+s}{B+s}$  again (note that $1-b\leq b(p-2)$ since $b \geq \frac{1}{p-1}$). The obtained inequality is the same as (\ref{we2}).  The lemma is proved.

 \end{proof} 

Clearly the lemma proves inequality (\ref{in3}) in the case $p\geq 3$. In particular, we just reproved the $p=q$ case of the theorem of Weissler, i.e.,  showed that the conjecture holds   for all $p \geq 1$ except $p \in (3/2,2)\cup(2,3)$. The reader may think that the argument presented in this section is different from the one of Weissler \cite{W1979} because, for example, Weissler uses non-trivial estimates for a certain implicitly defined function.  Nevertheless, we should say that both arguments are essentially the same because  they use the inequality (\ref{infl}) and the monotonicity expressed by the mock log-Sobolev inequality.  

Before we move to the case $p \in (2,3)$, let us explain in the next section that  the monotonicity approach  we just presented  cannot be adapted to that case.

\subsection{Why is the case  $p \in (2,3)$  difficult? Uniqueness lemma}

Weissler writes in his paper (see a remark on page 117 in \cite{W1979}) ``Even though Proposition 7 is false without the condition $p\geq 3$, one should not give hope for (3.9)''. Without going into the details, this remark says that the reason the monotonicity argument (\ref{map1}) fails when $p\in (2,3)$ is because the estimate (\ref{infl})  was too rough. One could hope that there might  be a better substitute for (\ref{infl}). However, we will now show that this is not the case and thus  one should give up on the  chase for ``monotone quantities'' when $p \in (2,3)$. 

Let $f \in C^{1}([0, \pi/2])$, $f>0$, be  such that 
\begin{align}
c(t) \geq \frac{f(a)}{f(a+t)} \quad \text{for all} \; 0\leq a \leq a+t\leq \pi/2 \quad \label{prikin0}
\end{align}
and the map 
\begin{align}
\psi(s) = \left(1+\frac{x^{2}}{f^{2}(s)}+\frac{2x\cos(s)}{f(s)} \right)^{p/2}+\left(1+\frac{x^{2}}{f^{2}(s)}-\frac{2x\cos(s)}{f(s)} \right)^{p/2} \label{prikin} 
\end{align}
is increasing on $[0, \pi/2]$ for all $x\geq 0$.
Clearly, as we have seen in the previous section, if such $f$ exists, then the two-point inequality (\ref{in3}) follows.  
\begin{lemma}[Uniqueness of the mock log-Sobolev inequality]\label{uniq}
If (\ref{prikin0}) holds  and the map given by (\ref{prikin}) is increasing on $[0, \pi/2]$ for all $x\geq 0$, then necessarily  $f(s) = C \sqrt{1+(p-2)\cos^{2}(s)}$ on $[0, \pi/2]$ for some constant $C>0$. 
\end{lemma}
In other words, the lemma says that one needs to come up with a different approach to prove (\ref{in3}) when $p\in (2,3)$. 

\begin{proof}

Notice that when $x \approx 0$, the map (\ref{prikin}) behaves as 
\begin{align*}
\psi(s) = 2+\frac{p(1+(p-2)\cos^{2}(s))}{f(s)^{2}}x^{2}+O(x^{4}).
\end{align*}
Therefore, the map  
\begin{align}\label{xashi}
h(s) = \frac{1+(p-2)\cos^{2}(s)}{f(s)^{2}}
\end{align}
 should be increasing on $[0, \pi/2]$. The latter together with (\ref{prikin0}) implies that 
 \begin{align}\label{utoloba}
 c(t)\geq \frac{f(a)}{f(a+t)}\geq \frac{\sqrt{1+(p-2)\cos^{2}(a)}}{\sqrt{1+(p-2)\cos^{2}(a+t)}}.
 \end{align}
 
 Next we claim that $h$ is constant on $[0,\pi/2]$. To prove the claim,  we notice that the monotonicity of $h$, i.e.,  the condition $h'(s)\geq 0$, can be written as 
 \begin{align}\label{bol44}
\frac{d}{ds} \ln(f(s)) \leq \frac{d}{ds} \ln (g(s))
 \end{align}
 where $g(s) = \sqrt{1+(p-2)\cos^{2}(s)}$. Integrating (\ref{bol44}) over the interval $[0,\pi/2]$ with respect to $s$, we obtain 
 \begin{align*}
 \frac{1}{c(\pi/2)}\stackrel{(\ref{utoloba})}{\leq} \frac{f(\pi/2)}{f(0)} \stackrel{(\ref{bol44})}{\leq} \frac{g(\pi/2)}{g(0)} = \frac{1}{\sqrt{p-1}} \stackrel{(\ref{defr})}{=}r(\pi/2) \stackrel{(\ref{maxc})}{=} \frac{1}{c(\pi/2)},
 \end{align*}
 which means that (\ref{bol44}) must be an equality for all $s \in (0, \pi/2)$ and, thereby, 
$f(s)=Cg(s)$ on $(0, \pi/2)$. The claim, and hence the lemma is proved.

\end{proof}

\subsection{Self-improvement and hidden invariance in the two-point inequality}

\begin{lemma}\label{dadis2}
It is enough to check  (\ref{in1}) for  $0\leq c(t)y \leq 1$. 
\end{lemma}
\begin{proof}
Assuming that the inequality (\ref{in1}) holds with some fixed $a$ and $t$ for all $y$ satisfying $0 \leq cy \leq 1$ where    $c=c(t)$, we show that it also holds with the same $a,t$ for the case when $cy >1$. Fix $y$ such that $cy>1$. 
 Dividing both sides of the inequality by $(cy)^{p}$,  we can rewrite (\ref{in1}) as 
\begin{align}\label{referi1}
\left| \frac{1}{cy}+e^{i(t+a)}\right|^{p}+\left| \frac{1}{cy}-e^{i(t+a)}\right|^{p} \geq \left| \frac{1}{cy}+e^{ia}\frac{1}{c}\right|^{p}+\left| \frac{1}{cy}-e^{ia}\frac{1}{c}\right|^{p}.
\end{align}
Using the identities $\left| \frac{1}{cy}\pm e^{i(t+a)}\right|^{p} = \left| \frac{1}{cy}e^{i(t+a)}\pm1\right|^{p}$ and $\left| \frac{1}{cy}\pm e^{ia}\frac{1}{c}\right|^{p} =\left| \frac{1}{cy}e^{ia}\pm \frac{1}{c}\right|^{p}$ we can rewrite (\ref{referi1}) as 
\begin{align*}
\left| \frac{1}{cy}e^{i(t+a)}+1\right|^{p}+\left| \frac{1}{cy}e^{i(t+a)}-1\right|^{p} \geq \left| \frac{1}{cy}e^{ia}+\frac{1}{c}\right|^{p}+\left| \frac{1}{cy}e^{ia}-\frac{1}{c}\right|^{p}.
\end{align*}
To verify the latter inequality let $\tilde{y} = \frac{1}{c^{2}y}$. Then  $c\tilde{y}<1$, i.e., we are in the range in which we assumed the validity of the estimate  (\ref{in1}) for the pair $c, \tilde{y}$. Applying (\ref{in1}) to $c, \tilde{y}$ we obtain  
\begin{align*}
\left| \frac{1}{cy}e^{i(t+a)}+1\right|^{p}+\left| \frac{1}{cy}e^{i(t+a)}-1\right|^{p}  \geq 
\left| \frac{1}{c^{2}y}e^{ia}+1\right|^{p}+\left| \frac{1}{c^{2}y}e^{ia}-1\right|^{p}.
\end{align*}
Next, we claim that 
\begin{align*}
\left| \frac{1}{c^{2}y}e^{ia}+1\right|^{p}+\left| \frac{1}{c^{2}y}e^{ia}-1\right|^{p} \geq 
\left| \frac{1}{cy}e^{ia}+\frac{1}{c}\right|^{p}+\left| \frac{1}{cy}e^{ia}-\frac{1}{c}\right|^{p}.
\end{align*}
Indeed, after multiplying both sides of the inequality by $c^{p}$, we can rewrite the latter estimate as        
\begin{align}\label{in2}
\left(c^{2}+\frac{1}{c^{2}y^{2}}+\frac{2}{y}\cos(a)\right)^{p/2}&+\left(c^{2}+\frac{1}{c^{2}y^{2}}-\frac{2}{y}\cos(a)\right)^{p/2}\geq \\
\left(1+\frac{1}{y^{2}}+\frac{2}{y}\cos(a)\right)^{p/2}&+\left(1+\frac{1}{y^{2}}-\frac{2}{y}\cos(a)\right)^{p/2}.\nonumber
\end{align}
Next, notice that 
\begin{align*}
c^{2}+\frac{1}{c^{2}y^{2}} -\left( 1+\frac{1}{y^{2}}\right) = \frac{(c^{2}-1)(c^{2}y^{2}-1)}{c^{2}y^{2}}\geq 0,
\end{align*}
where we have used the fact that $c \geq 1$ and $cy>1$. Therefore    (\ref{in2}) follows from the fact that the mapping 
$s \mapsto (s+A)^{p/2}+(s-A)^{p/2}$ is increasing on $[A, \infty)$.  Thus (\ref{in1}) holds for all $y \geq 0$. 
\end{proof}

\subsection{From multiplicativity to additivity:  chasing the fourth order terms}
Let $p \in (2,3)$. We only need to verify (\ref{in1}) in the regime when $0\leq a \leq a+t \leq \frac{\pi}{2}$ and $0\leq c y \leq 1$ where $c=c(t)$. Indeed, assuming that (\ref{in1}) is proved for such $a,t,y$, Lemma~\ref{dadis2} allows us to extend the range of $y$ to $[0, +\infty)$ keeping the restriction on $a,t$ only.  On the other hand  Lemma~\ref{dadis} says that if (\ref{in1}) holds for some $y \geq 0$ and all $a,t$ such that $0\leq a \leq a+t\leq \pi/2$ then it holds for the same $y$ and all $a,t$.

We would like to prove the inequality 
\begin{align}\label{in5}
(1+c^{2}y^{2} + 2cy \cos(a+t))^{s}&+(1+c^{2}y^{2} - 2cy \cos(a+t))^{s}\geq\\
 (1+y^{2} + 2y \cos(a))^{s}&+(1+y^{2} - 2y \cos(a))^{s}.\nonumber
\end{align}
where 
\begin{align}
s = \frac{p}{2}\in \left(1, \frac{3}{2}\right),  \quad c=c(t) = \frac{1}{r(t)} \in \left[1,\sqrt{p-1}\right]. \label{cic} 
\end{align}

Dividing both sides of (\ref{in5}) by $2(1+y^{2})^{s}$ and expanding both sides  into power series, we can rewrite (\ref{in5}) as 
\begin{align*}
\left(\frac{1+c^{2}y^{2}}{1+y^{2}}\right)^{s}\sum_{\ell=0}^{\infty} \left(\frac{2cy \cos(a+t)}{1+c^{2}y^{2}}\right)^{2\ell} \binom{s}{2\ell}\geq  \sum_{\ell=0}^{\infty} \left(\frac{2y \cos(a)}{1+y^{2}}\right)^{2\ell} \binom{s}{2\ell}.
\end{align*}
We can estimate the left hand side as  
\begin{align*}
&LHS = \left(\frac{1+c^{2}y^{2}}{1+y^{2}}\right)^{s}\sum_{\ell=0}^{\infty} \left(\frac{2cy \cos(a+t)}{1+c^{2}y^{2}}\right)^{2\ell} \binom{s}{2\ell}\geq \\
&\left(\frac{1+c^{2}y^{2}}{1+y^{2}}\right)^{s} + \left(\frac{1+c^{2}y^{2}}{1+y^{2}}\right)^{s}\left(\frac{2cy \cos(a+t)}{1+c^{2}y^{2}}\right)^{2} \binom{s}{2} + \sum_{\ell=2}^{\infty}\left(\frac{2cy \cos(a+t)}{1+c^{2}y^{2}}\right)^{2\ell} \binom{s}{2\ell}\geq \\
&\left(\frac{1+c^{2}y^{2}}{1+y^{2}}\right)^{s} + \left(\frac{1+c^{2}y^{2}}{1+y^{2}}\right)^{s}\left(\frac{2cy \cos(a+t)}{1+c^{2}y^{2}}\right)^{2} \binom{s}{2} + \sum_{\ell=2}^{\infty}\left(\frac{2y \cos(a+t)}{1+y^{2}}\right)^{2\ell} \binom{s}{2\ell}.
\end{align*}
In the first inequality we used the fact that $\frac{1+c^{2}y^{2}}{1+y^{2}}\geq 1$ and $\binom{s}{2\ell}\geq 0$. In the second inequality we used the fact that 
\begin{align*}
\frac{cy}{1+c^{2}y^{2}} - \frac{y}{1+y^{2}} = \frac{y(c-1)(1-cy^{2})}{(1+c^{2}y^{2})(1+y^{2})}\geq 0
\end{align*}
which is true because $0\leq cy\leq 1$ and $c\geq 1$.

The right hand side can be rewritten as 
\begin{align*}
RHS = 1 + \left(\frac{2y \cos(a)}{1+y^{2}}\right)^{2} \binom{s}{2} + \sum_{\ell=2}^{\infty}\left(\frac{2y \cos(a)}{1+y^{2}}\right)^{2\ell} \binom{s}{2\ell}.
\end{align*}

Thus, it suffices to prove the inequality 
\begin{align}\label{in6}
&\left(\frac{1+c^{2}y^{2}}{1+y^{2}}\right)^{s} + \left(\frac{1+c^{2}y^{2}}{1+y^{2}}\right)^{s}\left(\frac{2cy \cos(a+t)}{1+c^{2}y^{2}}\right)^{2} \binom{s}{2}-1-\left(\frac{2y \cos(a)}{1+y^{2}}\right)^{2} \binom{s}{2} \\
&\geq  \sum_{\ell=2}^{\infty}\left(\frac{2y}{1+y^{2}}\right)^{2\ell}\left(\cos^{2\ell}(a)-\cos^{2\ell}(a+t) \right) \binom{s}{2\ell}.\nonumber
\end{align}

\subsection{Contribution of the infinite series}
In this section we prove the following key lemma  which gives the upper bound for the infinite series on the right hand side of (\ref{in6}). 
\begin{lemma}\label{utol5}
We have 
\begin{align*}
&\sum_{\ell=2}^{\infty} \left(\frac{2y}{1+y^{2}}\right)^{2\ell}\left(\cos^{2\ell}(a) - \cos^{2\ell}(a+t)\right)\binom{s}{2\ell} \leq \\
&\frac{\sqrt{3}}{4}\cdot \frac{s(s-1)(s-2)(s-3)}{2} \left(\frac{2y}{1+y^{2}}\right)^{2} y^{2}\sin(t)
\end{align*}
for all $0\leq a \leq a+t\leq \pi/2$, $y\geq 0$, and  $s \in (1,3/2)$.
\end{lemma}
\begin{proof}
Let us denote $w = \frac{2y}{1+y^{2}}$. Then $y = \frac{1-\sqrt{1-w^{2}}}{w}$, and therefore 
\begin{align*}
\left(\frac{2y}{1+y^{2}}\right)^{2} y^{2} = (1-\sqrt{1-w^{2}})^{2} = 2-2\sqrt{1-w^{2}}-w^{2} =2 \sum_{k=2}^{\infty}w^{2k} \left|\binom{1/2}{k}\right|= \sum_{k=2}^{\infty}a_{k}w^{2k},
\end{align*}
where $a_{k} =  2\left|\binom{1/2}{k}\right|$ for $k \geq 2$. Clearly $a_{2} = \frac{1}{4}$ and 
\begin{align*}
\frac{a_{k+1}}{a_{k}} = \frac{\left|\binom{1/2}{k+1}\right|}{\left|\binom{1/2}{k}\right|} = \frac{k-\frac{1}{2}}{k+1}.
\end{align*}
Thus it suffices to show that 
\begin{align*}
\sum_{\ell=2}^{\infty} \left(\cos^{2\ell}(a) - \cos^{2\ell}(a+t) \right) \binom{s}{2\ell} w^{2\ell} \leq 
\frac{\sqrt{3}}{4}\cdot \frac{s(s-1)(s-2)(s-3)}{2} \sin(t) \sum_{\ell=2}^{\infty} a_{\ell} w^{2\ell}.
\end{align*}
We have 
\begin{align*}
\cos^{2\ell}(a) - \cos^{2\ell}(a+t) = \int_{a}^{a+t} - \frac{d}{dx} \cos^{2\ell}(x) dx  = 2\ell \int_{a}^{a+t}\cos^{2\ell-1}(x) \sin(x)dx. 
\end{align*}
By Lemma~\ref{ode} proved below, the right hand side can be estimated from above by
\begin{align*}
2\ell   \sup_{x \in \mathbb{R}}\left(  \cos^{2\ell-1}(x) \sin(x)\right) \cdot \sin(t)= \sqrt{2\ell} \left(\frac{2\ell-1}{2\ell}\right)^{\frac{2\ell-1}{2}}\, \cdot \sin(t).
\end{align*}
Therefore
\begin{align*}
&\sum_{\ell=2}^{\infty} \left(\cos^{2\ell}(a) - \cos^{2\ell}(a+t) \right) \binom{s}{2\ell} w^{2\ell} \leq \\
&\sin(t)\sum_{\ell=2}^{\infty}  \sqrt{2\ell} \left(\frac{2\ell-1}{2\ell}\right)^{\frac{2\ell-1}{2}} \binom{s}{2\ell} w^{2\ell}=\sin(t)\sum_{\ell=2}^{\infty} b_{\ell} w^{2\ell},
\end{align*}
where $b_{\ell}  =  \sqrt{2\ell} \left(\frac{2\ell-1}{2\ell}\right)^{\frac{2\ell-1}{2}} \binom{s}{2\ell}$ for $\ell\geq 2$.  We have 
\begin{align*}
\frac{b_{\ell+1}}{b_{\ell}} = \frac{\sqrt{2(\ell+1)} \left(\frac{2\ell+1}{2\ell+2}\right)^{\frac{2\ell+1}{2}} \binom{s}{2\ell+2}}{ \sqrt{2\ell} \left(\frac{2\ell-1}{2\ell}\right)^{\frac{2\ell-1}{2}} \binom{s}{2\ell}}=\sqrt{\frac{\ell+1}{\ell}}\cdot  \frac{(2\ell-s)(2\ell+1-s)}{(2\ell+1)(2\ell+2)}\cdot \frac{ \left(\frac{2\ell+1}{2\ell+2}\right)^{\frac{2\ell+1}{2}}}{\left(\frac{2\ell-1}{2\ell}\right)^{\frac{2\ell-1}{2}}}.
\end{align*}
We claim that $\frac{b_{\ell+1}}{b_{\ell}} \leq \frac{\ell-\frac{1}{2}}{\ell+1}$. Indeed,  $\sqrt{\frac{1+\ell}{\ell}} \leq 1+\frac{1}{2\ell}$ and, since the mapping $n \mapsto (1-1/n)^{n-1}$ is decreasing for $n >1$, we have $ \left(\frac{2\ell+1}{2\ell+2}\right)^{\frac{2\ell+1}{2}} \leq \left(\frac{2\ell-1}{2\ell}\right)^{\frac{2\ell-1}{2}}$. Next, we notice that 
 $(2\ell-s)(2\ell+1-s) \leq (2\ell-1)2\ell$. Finally, combining these three estimates,  we obtain that 
 \begin{align*}
\frac{b_{\ell+1}}{b_{\ell}}=\sqrt{\frac{\ell+1}{\ell}}\cdot  \frac{(2\ell-s)(2\ell+1-s)}{(2\ell+1)(2\ell+2)}\cdot \frac{ \left(\frac{2\ell+1}{2\ell+2}\right)^{\frac{2\ell+1}{2}}}{\left(\frac{2\ell-1}{2\ell}\right)^{\frac{2\ell-1}{2}}} \leq \left(1+\frac{1}{2\ell} \right) \frac{(2\ell-1)2\ell}{(2\ell+1)(2\ell+2)} = \frac{\ell-\frac{1}{2}}{\ell+1}. 
\end{align*}
Telescoping the product of  $\frac{b_{k+1}}{b_{k}} \leq \frac{a_{k+1}}{a_{k}}$ ($k=2, \ldots, \ell)$, we obtain 
\begin{align*}
b_{\ell+1} \leq a_{\ell+1}  \frac{b_{2}}{a_{2}} = a_{\ell+1} 8 \left(\frac{3}{4}\right)^{\frac{3}{2}}\binom{s}{4} = a_{\ell+1} \frac{s(s-1)(s-2)(s-3)}{2}  \cdot \frac{\sqrt{3}}{4}.
\end{align*}
Therefore 
\begin{align*}
&\sin(t)\sum_{\ell=2}^{\infty} b_{\ell} w^{2\ell}\leq \sin(t) \cdot  \frac{s(s-1)(s-2)(s-3)}{2}  \cdot \frac{\sqrt{3}}{4} \sum_{\ell=2}^{\infty} a_{\ell} w^{2\ell}=\\
&\sin(t)  \cdot \frac{s(s-1)(s-2)(s-3)}{2}  \cdot \frac{\sqrt{3}}{4} \cdot \left(\frac{2y}{1+y^{2}}\right)^{2} y^{2}.
\end{align*}
Thus it remains to prove the following lemma.
\end{proof}
\begin{lemma}\label{ode}
For all $0\leq a\leq a+t\leq \frac{\pi}{2}$ and all $\ell\geq 2$, we have 
\begin{align*}
\int_{a}^{a+t}\cos^{2\ell-1}(x)\sin(x)dx \leq \sup_{x\in \mathbb{R}}\left(  \cos^{2\ell-1}(x) \sin(x)\right) \cdot \sin(t).
\end{align*}
\end{lemma}
\begin{proof}
First, we need the following 
\begin{lemma}[Cap lemma]\label{caple}
Let $f$ and $g$ be two continuous unimodal nonnegative real valued functions  defined on $\mathbb{R}$ such that $f=0$ on $ \mathbb{R}\setminus (a,b)$, and $g=0$ on $\mathbb{R}\setminus (a',b')$.  Assume that $a'\leq a<b'\leq b$ and $x_{0} \in (a,b')$ is the point of the common global maximum of  $f$ and $g$ with  $f(x_{0})=g(x_{0})$. Suppose also that there exists $c \in (x_{0}, b')$ such that  $g(x)\geq f(x)$ on $[a',c]$ and $g(x)\leq f(x)$ on $[c,b]$ (see Fig.~\ref{shapka}).
\begin{figure}[ht]
\centering
\includegraphics[scale=1]{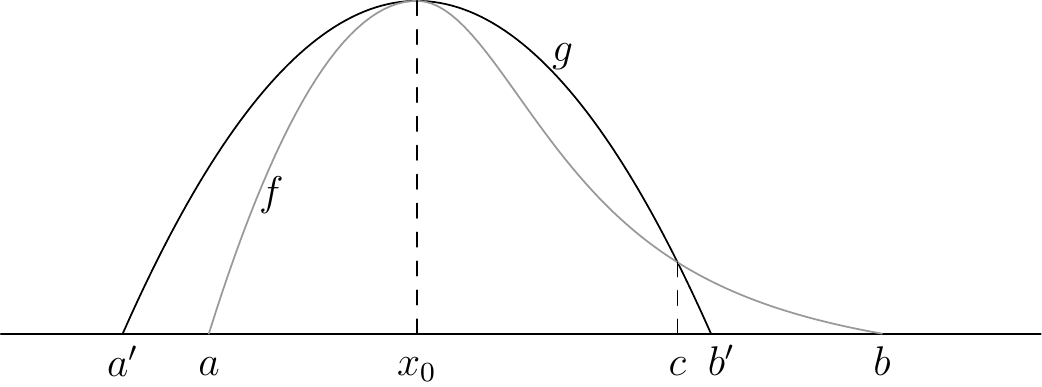}
\caption{The functions $f$ and $g$.}
\label{shapka}
\end{figure}
\par\noindent
At last assume that 
\begin{align}\label{ploshad}
\int_{x_{0}}^{b'}g \geq \int_{x_{0}}^{b}f.
\end{align}
 Then for all $t \in [0, b-a]$, we have
\begin{align}\label{vylez}
\max_{\substack{|I|=t\\ I\subset [a,b] \text{ is an interval}}} \int_{I}f \leq \max_{\substack{|I|=t \\ I \subset \mathbb{R} \text{ is an interval}}} \int_{I}g.
\end{align} 
\end{lemma}

\begin{proof}
It follows from the unimodality of $f$ that the maximum on the left hand side of (\ref{vylez}) is attained when  $I=[\alpha, \beta]$ with $\alpha \leq x_{0}$ and $\beta \geq x_{0}$. Consider two cases.  If $\beta \in [x_{0}, c]$, then there is nothing to prove because $f \leq g$ on $I$.  If $\beta \geq  c$, then we have 
\begin{align*}
\int_{I}f &= \int_{\alpha}^{x_{0}} f + \int_{x_{0}}^{\beta} f  \leq \int_{\alpha}^{x_{0}} g + \int_{x_{0}}^{\beta} f  = \int_{I}g + \int_{x_{0}}^{c}(f-g)  + \int_{c}^{\beta}(f-g) \\
&\leq \int_{I}g + \int_{x_{0}}^{c}(f-g)  + \int_{c}^{b}(f-g) = \int_{I}g +\int_{x_{0}}^{b} (f-g) \stackrel{(\ref{ploshad})}{\leq}  \int_{I}g.
\end{align*}
The second inequality follows from the fact that $f\geq g$ on $[\beta, b]$. Lemma~\ref{caple} is proved.
\end{proof}
Next, fix any integer $\ell\geq 2$. Take $a=0$, $b=\frac{\pi}{2}$, $f(x) = \cos^{2\ell-1}(x) \sin(x)$, $x_{0} = \arcsin\frac{1}{\sqrt{2\ell}}$. Redefine $f$ to be $0$ outside  $[0, \pi/2]$. To construct an appropriate $g$, we calculate the derivatives of $f$. For $x \in (0, \pi/2)$, we have 
\begin{align}
f'(x) &= -(2\ell-1)\cos^{2\ell-2}(x) \sin^{2}(x) + \cos^{2\ell}(x) =-(2\ell-1)\cos^{2\ell-2}(x)+2\ell \cos^{2\ell}(x) ;\nonumber\\
f''(x) &= (2\ell-1)(2\ell-2)\cos^{2\ell-3}(x) \sin(x) -4\ell^{2}\cos^{2\ell-1}(x) \sin(x)  \nonumber\\
& = f(x)\left(\frac{(2\ell-1)(2\ell-2)}{\cos^{2}(x)} - 4\ell^{2}\right). \label{der33}
\end{align}
In particular, we see that $f''/f$ is increasing on $[0, \pi/2]$. We have 
\begin{align*}
\frac{f''(x_{0})}{f(x_{0})}=\frac{(2\ell-1)(2\ell-2)}{1-\frac{1}{2\ell}} - 4\ell^{2} = -4\ell.
\end{align*}
This suggests that we should take 
\begin{align*}
g(x) = 
\begin{dcases}
f(x_{0}) \cos (2\sqrt{\ell} (x_{0}-x)), & x\leq x_{0},\\
f(x_{0}) \cos(A_{\ell}(x-x_{0})), & x\geq x_{0},
\end{dcases}
\end{align*}
where $A_{\ell}$ satisfies $\frac{1}{A_{\ell}} + \frac{1}{2\sqrt{\ell}} =1$. Next, let $a'\leq x_{0}$ be the largest number such that $g(a')=0$, i.e., $a' = x_{0} - \frac{\pi}{2 \cdot 2\sqrt{\ell}}$. Let $b'\geq x_{0}$ be the smallest number such that $g(b')=0$, i.e., $b' = \frac{\pi}{2 A_{\ell}} + x_{0}$. Redefine $g$ to be zero outside $(a',b')$. 

Note that  by the choice of $A_{\ell}$, $g$ is equimeasurable with the mapping $s \mapsto f(x_{0})\cos(s)$,  $s \in [0, \pi/2]$, i.e., 
$$
| \{x \in \mathbb{R}\, :\,  g(x)>\lambda\}| = |\{s \in [0,\pi/2] :\, f(x_{0})\cos(s)> \lambda \}|
$$
 for all $\lambda > 0$. Thereby, for every $t \in (0, \pi/2)$, we have 

\begin{align}\label{maxa}
\max_{\substack{|I|=t \\ I \text{\, is an interval }}} \int_{I}g &\leq \max_{\substack{|E|=t \\ E \text{\, is measurable}}}  \int_{E} g=\max_{\substack{|E'|=t, \, E'\subset [0, \pi/2] \\ E' \text{\, is measurable}}}  \int_{E'} f(x_{0})\cos(s)ds\\
&  = \int_{0}^{t} f(x_{0}) \cos(s)ds = f(x_{0})\sin(t).\nonumber
\end{align}

\begin{lemma}\label{vspom}
Functions $f$ and $g$ satisfy the conditions of the cap lemma.
\end{lemma}

\begin{proof}
Clearly both $f$ and $g$ are unimodal functions, $x_{0}$ is the point of the global maximum for $f$ and $g$, and  $f(x_{0})=g(x_{0})$. Since $\arcsin(s)< \frac{\pi}{2}s$ for every $s \in (0,1)$, we  conclude that  $a' =\arcsin(\frac{1}{2\sqrt{\ell}}) -\frac{\pi}{2} \cdot \frac{1}{2\sqrt{\ell}}<0$. The choice of $A_{\ell}$ implies that  $b'-a'=\pi/2$. Thereby $a'<0=a<x_{0}<b'<\pi/2=b$.

 Next, we need to check that 
\begin{align*}
&\int_{x_{0}}^{b'}g = f(x_{0}) \frac{1}{A_{\ell}} = \frac{1}{\sqrt{2\ell}} \left(\frac{2\ell-1}{2\ell}\right)^{\frac{2\ell-1}{2}}\left(1-\frac{1}{2\sqrt{\ell}} \right)\geq \\
&\int_{x_{0}}^{\pi/2}f = \left.\frac{1}{2\ell} \cos^{2\ell}(x)\right|_{\frac{\pi}{2}}^{x_{0}} = \frac{1}{2\ell}\left(\frac{2\ell-1}{2\ell}\right)^{\ell},
\end{align*}
i.e., that $1-\frac{1}{2\sqrt{\ell}} \geq \frac{1}{\sqrt{2\ell}} \sqrt{1-\frac{1}{2\ell}}$, 
which is indeed true even  for $\ell \geq 1$.

In order to show that $g(x) \geq f(x)$ for $x \in [a', x_{0}]$, it suffices to check the claim  $g(x)\geq f(x)$ on $[0, x_{0}]$. The claim follows from the fact that $f(x_{0})=g(x_{0})>0$, $f'(x_{0})=g'(x_{0})=0$, and $\frac{f''}{f}<\frac{g''}{g}$ on $[0,x_{0})$. Indeed, we calculate 
\begin{align*}
&\lim_{t \to x_{0}^{-}}g''(t) =f''(x_{0}) = -4\ell f(x_{0});\\ 
&f'''(x_{0})=f(x_{0})\left(\frac{2(2\ell-1)(2\ell-2)\sin(x_{0})}{\cos^{3}(x_{0})}\right) >0=\lim_{t \to x_{0}^{-}}g'''(t)
\end{align*}
(to calculate $f'''(x_{0})$ quickly, use (\ref{der33}) and the fact that $f'(x_{0})=0$). 
Therefore $g(x)>f(x)$ when $x \in (x_{0}-\varepsilon, x_{0})$ provided that $\varepsilon>0$ is sufficiently small. It follows from the piece-wise analyticity of $f$ and $g$ that the equation $f(x)=g(x)$ can have only finite number of solutions on $(0,x_{0})$. Let $x_{1}\in (0,x_{0})$ be the largest point (if it exists) such that $g\geq f$ on $(x_{1},x_{0})$ and $g<f$ on $(x_{1}-\delta, x_{1})$ for a sufficiently small $\delta>0$.  Clearly $f(x_{1})=g(x_{1})$ and $g'(x_{1})\geq f'(x_{1})$, so we have
\begin{align*}
0\leq (f'g-fg')|_{x_{1}}^{x_{0}} = \int_{x_{1}}^{x_{0}}\left(\frac{f''}{f}-\frac{g''}{g}\right)fg <0,
\end{align*}
which is a contradiction.  Thus  there is no such $x_{1}$, which implies that  $g\geq f$ on $(0,x_{0})$ and we are done.

Next, we show that there exists $c \in (x_{0}, b')$ such that $g\geq f$ on $(x_{0},c)$ and $g\leq f$ on $(c,b')$ (on $[b',\pi/2)$ we clearly have $g=0\leq f$).  Note that 
\begin{align*}
\lim_{t\to x_{0}^{+}} g''(t) = -f(x_{0})A_{\ell}^{2} = -f(x_{0})\frac{4\ell}{(2\sqrt{\ell}-1)^{2}}> -4\ell f(x_{0}) = f''(x_{0})
\end{align*}
for $\ell\geq 2$.  Thus  $g>f$ on $(x_{0},x_{0}+\varepsilon)$ provided that $\varepsilon>0$ is sufficiently small. By the piece-wise  analyticity of $f$ and $g$, the equation  $f(x)=g(x)$ has finite number of solutions on $[x_{0}, b')$. Let $x_{1}>x_{0}$ be the smallest number  such that $g\geq f$ on $(x_{0}, x_{1})$ and $g<f$ on $(x_{1}, x_{1}+\delta)$ for a sufficiently small $\delta>0$. If there were no such point, we would have $g\geq f$ on $(x_{0},b']$ and,  in particular, $0=g(b')\geq f(b')>0$, which is a  contradiction. 

If the inequality  $g\leq f$ is violated on $[x_{1},b']$, there exists a point $x_{2}\in (x_{1},b')$ such that $g\leq f$ on $(x_{1},x_{2})$ and $g>f$ on $(x_{2}, x_{2}+\delta')$ for some sufficiently small $\delta'>0$. 

Note that $f(x_{1})=g(x_{1})$ and $f'(x_{1})\geq g'(x_{1})$, so 
\begin{align*}
0\leq (f'g-fg')|_{x_{0}}^{x_{1}} = \int_{x_{0}}^{x_{1}}\left(\frac{f''}{f}-\frac{g''}{g}\right)fg,
\end{align*}
whence $\frac{f''}{f}-\frac{g''}{g}\geq 0$ somewhere on $[x_{0}, x_{1}]$ and, thereby, $\frac{f''}{f}-\frac{g''}{g}>0$  on $(x_{1},x_{2})$ (since $f''/f$ is strictly increasing and $g''/g$ is constant). On the other hand, we have  $f(x_{2})=g(x_{2})$, $f'(x_{2})\leq g'(x_{2})$ and therefore 
\begin{align*}
0\geq (f'g-fg')|_{x_{1}}^{x_{2}} = \int_{x_{1}}^{x_{2}} \left(\frac{f''}{f} - \frac{g''}{g}\right)fg >0,
\end{align*}
which is a  contradiction.  Thus we can take $c=x_{1}$. 
\end{proof}
Lemma~\ref{ode} is now completely proved. 
\end{proof}
\subsection{Sharpening Bernoulli}
Combining Lemma~\ref{utol5} and inequality (\ref{in6}), we see that it suffices to prove the inequality 
\begin{align}\label{utol8}
&\left(\frac{1+c^{2}y^{2}}{1+y^{2}}\right)^{s}+\left(\frac{1+c^{2}y^{2}}{1+y^{2}}\right)^{s}\left(\frac{2cy\cos(a+t)}{1+c^{2}y^{2}}\right)^{2}\binom{s}{2} -1 - \left(\frac{2y\cos(a)}{1+y^{2}}\right)^{2}\binom{s}{2} \geq \\
&\frac{\sqrt{3}}{4}\cdot \frac{s(s-1)(s-2)(s-3)}{2} \left(\frac{2y}{1+y^{2}}\right)^{2} y^{2}\sin(t)\nonumber
\end{align}
for all $0\leq y \leq \frac{1}{c}$, $0\leq a \leq a+t\leq \pi/2$, $s \in (1,3/2)$, where $c=c(t)$ is defined by (\ref{cic}).

Let us estimate the left hand side from below. We have 
\begin{align*}
&LHS = \left(\frac{1+c^{2}y^{2}}{1+y^{2}}\right)^{s} -1 + \binom{s}{2}\left(\frac{1+c^{2}y^{2}}{1+y^{2}}\right)^{s-1} \frac{4c^{2}y^{2}\cos^{2}(a+t)}{(1+c^{2}y^{2})(1+y^{2})} - \binom{s}{2}\frac{4y^{2}\cos^{2}(a+t)}{(1+y^{2})^{2}} +\\
 &\binom{s}{2} \frac{4y^{2}}{(1+y^{2})^{2}}\left(\cos^{2}(a+t) -\cos^{2}(a) \right).
\end{align*}
Consider the map 
\begin{align*}
h(x) = x^{s}-1+\rho x^{s-1}, \quad x>1,
\end{align*}
where $\rho\in [0,1)$. Clearly $h''(x) = (s-1)x^{s-3}(sx+\rho(s-2))>0$ when $x\geq 1$. Therefore $h$ is convex on $[1, \infty)$ and hence  $h(x) \geq \rho+(s+\rho(s-1))(x-1)$ there. Let us apply the latter inequality to the case when $x = \frac{1+c^{2}y^{2}}{1+y^{2}}\geq 1$ and 
$$
\rho = \binom{s}{2} \frac{4c^{2}y^{2}\cos^{2}(a+t)}{(1+c^{2}y^{2})(1+y^{2})}<\frac{3/2\cdot 1/2}{2} \cdot \frac{2cy}{1+c^{2}y^{2}}\cdot  \frac{2y}{1+y^{2}}\cdot c \cdot \cos^{2}(a+t)\leq \frac{3}{8}\sqrt{2}<1. 
$$
Then we can estimate 
\begin{align*}
&LHS\geq \binom{s}{2} \frac{4c^{2}y^{2}\cos^{2}(a+t)}{(1+c^{2}y^{2})(1+y^{2})} + \frac{y^{2}(c^{2}-1)}{1+y^{2}}\left(s+(s-1)\binom{s}{2} \frac{4c^{2}y^{2}\cos^{2}(a+t)}{(1+c^{2}y^{2})(1+y^{2})} \right)-\\
&\binom{s}{2}\frac{4y^{2}\cos^{2}(a+t)}{(1+y^{2})^{2}}+\binom{s}{2} \frac{4y^{2}}{(1+y^{2})^{2}}\left(\cos^{2}(a+t) -\cos^{2}(a) \right) = \\[20pt] 
&s \frac{y^{2}(c^{2}-1)}{1+y^{2}}+\binom{s}{2}\left[ \left( 1+(s-1) \frac{y^{2}(c^{2}-1)}{1+y^{2}} \right) \frac{4c^{2}y^{2}\cos^{2}(a+t)}{(1+c^{2}y^{2})(1+y^{2})} - \frac{4y^{2}\cos^{2}(a+t)}{(1+y^{2})^{2}}\right]+\\
&\binom{s}{2} \frac{4y^{2}}{(1+y^{2})^{2}}\left(\cos^{2}(a+t) -\cos^{2}(a) \right) = \\[20pt] 
&s \frac{y^{2}(c^{2}-1)}{1+y^{2}}+\binom{s}{2}\left[ (s-1)y^{2}c^{2} \frac{4y^{2}\cos^{2}(a+t) (c^{2}-1)}{(1+c^{2}y^{2})(1+y^{2})^{2}} + \frac{4y^{2}\cos^{2}(a+t) (c^{2}-1)}{(1+c^{2}y^{2})(1+y^{2})^{2}}\right]+\\
&\binom{s}{2} \frac{4y^{2}}{(1+y^{2})^{2}}\left(\cos^{2}(a+t) -\cos^{2}(a) \right) =\\[20pt] 
&s \frac{y^{2}(c^{2}-1)}{1+y^{2}}+\binom{s}{2}\left[ \left( (s-1)y^{2}c^{2} +1\right) \frac{4y^{2}\cos^{2}(a+t) (c^{2}-1)}{(1+c^{2}y^{2})(1+y^{2})^{2}} \right]+\\
&\binom{s}{2} \frac{4y^{2}}{(1+y^{2})^{2}}\left(\cos^{2}(a+t) -\cos^{2}(a) \right) =\\[20pt] 
&s \frac{y^{2}(c^{2}-1)}{(1+y^{2})^{2}}\left\{ 1+y^{2}+ \frac{(s-1)\left( (s-1)y^{2}c^{2} +1\right)}{2}  \frac{4\cos^{2}(a+t) }{1+c^{2}y^{2}} \right\}+\\
&\frac{s}{2}(s-1) \frac{4y^{2}}{(1+y^{2})^{2}}\left(\cos^{2}(a+t) -\cos^{2}(a) \right) = \\[20pt]
&\frac{s}{2}\left(\frac{2y}{1+y^{2}}\right)^{2}\times \Bigg{\{} \frac{1}{2}  (c^{2}-1) \left( 1+2(s-1)\cos^{2}(a+t) \right)- (s-1) (\cos^{2}(a)-\cos^{2}(a+t))\\
&+\frac{1}{2}(c^{2}-1)y^{2}\left(1-\frac{2(s-1)(2-s)c^{2}\cos^{2}(a+t)}{1+c^{2}y^{2}} \right)\Bigg{\}}.
\end{align*}
Combining the obtained lower bound and  inequality (\ref{utol8}), we see that  it suffices to show that 
\begin{align}\label{utol9}
&\frac{1}{2}  (c^{2}-1) \left( 1+2(s-1)\cos^{2}(a+t) \right)- (s-1) (\cos^{2}(a)-\cos^{2}(a+t))+\\
&\frac{1}{2}(c^{2}-1)y^{2}\left(1-\frac{2(s-1)(2-s)c^{2}\cos^{2}(a+t)}{1+c^{2}y^{2}} \right)\geq \frac{\sqrt{3}}{4}\cdot (s-1)(s-2)(s-3) y^{2}\sin(t)\nonumber
\end{align}
for all $0\leq y\leq \frac{1}{c}, 0\leq a \leq a+t\leq \frac{\pi}{2}$, and  $s \in (1, 3/2)$.

\subsection{Moving to the boundary and factoring: an interplay between Analysis and Algebra}
We denote $c^{2}=C\geq 1$. We multiply both sides of inequality  (\ref{utol9}) by $2$ and estimate the factor $-\frac{1}{1+c^{2}y^{2}}$ on the left hand side of (\ref{utol9}) from below by $-1$. After rearranging the terms, we see that it suffices to show the inequality 
\begin{align}\label{utol13}
 (C-1) \left( 1+2(s-1)\cos^{2}(a+t) \right)- 2(s-1) (\cos^{2}(a)-\cos^{2}(a+t))&+\\
y^{2}\times\left\{ (C-1)\left(1-2(s-1)(2-s)C\cos^{2}(a+t) \right) -  \frac{\sqrt{3}}{2}\cdot (s-1)(s-2)(s-3)\sin(t)\right\} &\geq 0.\nonumber
\end{align}
The left hand side of (\ref{utol13}) is linear in $u=y^{2} \in [0,\frac{1}{C}]$. If $y^{2}=0$, the inequality reduces to 
\begin{align*}
C-1 \geq \frac{2(s-1) (\cos^{2}(a)-\cos^{2}(a+t))}{1+2(s-1)\cos^{2}(a+t)},
\end{align*}
which, after adding $1$ to both sides, reduces to  (\ref{infl}). Therefore, by linearity it suffices to consider  the case $y^{2}=\frac{1}{C}$. After substituting $y^{2}=\frac{1}{C}$, we can rewrite the left hand side of the inequality (\ref{utol13}) as 
\begin{align*}
& (C-1) \left( 1+2(s-1)\cos^{2}(a+t) \right)- 2(s-1) (\cos^{2}(a)-\cos^{2}(a+t))+\\
& (C-1)\left(\frac{1}{C}-2(s-1)(2-s)\cos^{2}(a+t) \right) -  \frac{\sqrt{3}}{2C}\cdot (s-1)(s-2)(s-3)\sin(t) = \\[20pt]
&(C-1)\left(1+\frac{1}{C}\right) + 2(C-1)(s-1)^{2} \cos^{2}(a+t)-2(s-1)(\cos^{2}(a)-\cos^{2}(a+t))\\
&-\frac{\sqrt{3}}{2C}\cdot (s-1)(s-2)(s-3)\sin(t)\geq \\[20pt]
&(C-1)\left(1+\frac{1}{C}\right)-2(s-1)(\cos^{2}(a)-\cos^{2}(a+t)) - \frac{\sqrt{3}}{2C}\cdot (s-1)(s-2)(s-3)\sin(t).
\end{align*}
Next, notice that $\cos^{2}(a)-\cos^{2}(a+t) = \sin(t)\sin(2a+t)\leq  \sin(t)$. Therefore it suffices to show the inequality 
\begin{align}\label{fin01}
(C-1)\left(1+\frac{1}{C}\right)\geq \left(2 + \frac{\sqrt{3}}{2C}\cdot (s-2)(s-3) \right)  (s-1)\sin(t).
\end{align}
It follows from (\ref{lens}) and the cosine theorem (see Fig.~\ref{fig:dom}) that
\begin{align*}
r(t)^{2}+\left(\frac{s-1}{\sqrt{2s-1}}\right)^{2}+2r(t)\frac{s-1}{\sqrt{2s-1}}\sin(t) = \frac{s^{2}}{2s-1}.
\end{align*} 
Using the equality $C = \frac{1}{r(t)^{2}}$, we obtain 
\begin{align*}
(s-1)\sin(t) = \frac{1-r(t)^{2}}{2r(t)} \sqrt{2s-1}=\frac{(C-1)\sqrt{2s-1}}{2\sqrt{C}}, \quad C \in [1, 2s-1]. 
\end{align*}
Therefore, inequality (\ref{fin01}) simplifies to 
\begin{align*}
\sqrt{C} \left(1+\frac{1}{C}\right) \geq \left[1+\frac{\sqrt{3}}{4C}(s-2)(s-3)\frac{}{}\right] \sqrt{2s-1} ,\quad C \in [1,2s-1], \quad s \in [1,3/2].
\end{align*}
Since $\sqrt{C}\left(1+\frac{1}{C}\right) \geq \sqrt{C} \frac{2}{\sqrt{C}}=2$,  $\frac{1}{C}\leq 1$, and $\sqrt{2s-1}\leq s$, it suffices to show that
\begin{align*}
2\geq \left[1+\frac{\sqrt{3}}{4}(s-2)(s-3)\right]s. 
\end{align*}
Subtracting $s$ from both sides of this inequality and dividing by $(2-s)$, we get $1 \geq \frac{\sqrt{3}}{4}s(3-s)$, i.e., 
$s^{2}-3s+\frac{4}{\sqrt{3}} = (s-\frac{3}{2})^{2} +(\frac{4}{\sqrt{3}} - \frac{9}{4})\geq 0$ to prove. It remains to notice that $\sqrt{3} \leq \frac{16}{9}$, i.e., 
\begin{align}
3\leq \frac{256}{81}.
\end{align}   

\subsection*{Acknowledgments} The authors are grateful to both anonymous referees for helpful suggestions. 
P.I.  was partially supported by the NSF grants DMS-1856486 and CAREER-DMS-2052865. F.N.  was partially supported by the NSF grant DMS-1600239.

\end{document}